\documentclass[a4paper,12pt]{lms}

\classno{05A15; Secondary 05A05, 11A25, 20B30}

\simpleequations

\usepackage{amssymb,amsmath,amsopn}

\usepackage{graphics}

\usepackage[all]{xy}

\renewcommand{\theta}{\vartheta}
\newtheorem{theorem}{Theorem}
\newtheorem{lemma}[theorem]{Lemma}

\newcommand{\Z}{\mathbf{Z}}

\DeclareMathOperator{\Cent}{Cent}

\newcommand{\supp}{\mathrm{supp}}

\begin{document}
\author{John R. Britnell}
\title{A formal identity involving\\ commuting triples of permutations}
\shorttitle{A formal identity}
\date{22 March 2012}

\begin{abstract}
We prove a formal power series identity, relating the arithmetic sum-of-divisors function to commuting triples of permutations. 
This establishes a conjecture of Franklin T. Adams-Watters.
\end{abstract}

\maketitle

The object of this note is to establish the following formal identity:
\begin{equation}\label{identity}
\prod_{j=1}^{\infty} \left(1-u^j\right)^{-\sigma(j)} = \sum_{n=0}^{\infty} \frac{T(n)}{n!} u^n,
\end{equation}
where $\sigma$ is the arithmetic sum-of-divisors function, and~$T(n)$ is the number of triples of pairwise-commuting elements of the symmetric group~$S_n$. (Here~$S_0$ is the trivial group.)
This is a surprising fact, as there seems no obvious reason for any connection between the function~$\sigma$ and commuting permutations.

The power series expansion of the left-hand side of this identity has coefficients which are listed on the Online Encyclopedia of Integer Sequences (OEIS) \cite{OEIS} as sequence A061256.
The coefficients on the right-hand side are listed as sequence A079860. The identity of the two sequences has been stated conjecturally on OEIS.
This conjecture (from 2006) is due to Franklin T. Adams-Watters; he informs me that it was based empirically on the numerical evidence.

For a finite group $G$, we shall write~$k(G)$ for the number of conjugacy classes of~$G$.
The following simple fact seems first to have been stated by Erd\H{o}s and Tur\'an \cite{ErdosTuran}.
\begin{lemma}\label{pairs}
The number of pairs of commuting elements of $G$ is $|G|\,k(G)$.
\end{lemma}

Let $g\in G$. It follows from Lemma \ref{pairs} that the number of commuting triples of~$G$ whose first element is~$g$, is given by $|\Cent_G(g)|\,k(\Cent_G(g))$.
So if~$T(G)$ is the total number of commuting triples, then
\begin{equation}\label{formula}
\frac{T(G)}{|G|} = \sum_{g \in G} \frac{|\Cent_G(g))|}{|G|} k(\Cent_G(g)) = \sum_{i=1}^r k(\Cent_G(g_i)),
\end{equation}
where $\{g_1,\dots, g_r\}$ is a set of conjugacy class representatives for~$G$.

In the case that~$G$ is the symmetric group~$S_n$, the conjugacy classes are parameterized by partitions of~$n$, whose parts correspond to cycle lengths. Let $g\in S_n$ have~$m_t$ cycles of length~$t$ for all~$t$. Then the centralizer of~$g$ in~$S_n$ is given (up to isomorphism) by
\[
\Cent_{S_n}(g) \cong \prod_{t=1}^n W(t,m_t),
\]
where $W(t,m)$ is the wreath product $\Z_t\wr S_m$. (Here~$\Z_t$ is used as a shorthand for $\Z/t\Z$, the integers modulo~$t$.)
It follows that
\begin{equation}\label{product}
k(\Cent_{S_n}(g)) = \prod_{t=1}^n k(W(t,m_t)).
\end{equation}

We may regard an element of $W(t,m)$ as a pair $(A,e)$, where $A\in {\Z_t}^m$ and $e\in S_m$. There is a natural action of~$S_m$ on the coordinates of~${\Z_t}^m$ given by $(B^e)_i=B_{ie^{-1}}$.
The group multiplication~$*$ in $W(t,m)$ is defined by
\[
(A,e) * (B,f) = (A+B^e,ef).
\]

Conjugacy in groups of the form $H \wr S_m$ is described in \cite[Section 4.2]{JamesKerber}; the case that $H=\Z_t$ is relatively straightforward.
Let $(A,e)$ be an element of $W(t,m)$, where  $A=(a_1,\dots,a_m)$.
Let~$c$ be a cycle of the permutation~$e$, and let $\supp(c)$ be the support of~$c$ (i.e.\ the elements of $\{1,\dots,m\}$ moved by~$c$).
We shall write~$|c|$ for $|\supp(c)|$, the length of the cycle.
Define the \emph{cycle~sum} $A[c] \in \Z_t$ by
\[
A[c] = \sum_{i \in \supp(c)} \!a_i.
\]
The \emph{cycle sum invariant} of $(A,e)$ corresponding to the cycle~$c$ is defined to be the pair $(A[c],|c|)$. The element $(A,e)$ has one such invariant for each cycle of~$e$.

\begin{lemma}\label{wpconj}
Two elements $(A,e)$ and $(B,f)$ of $W(t,m)$ are conjugate in $W(t,m)$ if and only if they have the same cycle sum invariants---that is, if and only if there is a bijection~$\tau$ between the cycles
of~$e$ and the cycles of~$f$, such that for any cycle~$c$ of~$e$ we have $(A[c],|c|) = (B[c\tau],|c\tau|)$.
\end{lemma}
\begin{proof}
See \cite[Theorem 4.2.8]{JamesKerber}, of which this is a particular case.
\end{proof}

Let $(A,e)$ be an element of $W(t,m)$. For each $z\in\Z_t$ we define~$\lambda_z$ to be the partition such that the multiplicity of~$\ell$ as a part of~$\lambda_z$ is equal to
the multiplicity of $(z,\ell)$ as a cycle sum invariant of $(A,e)$. Lemma \ref{wpconj} tells us that the partitions~$\lambda_z$ for $z\in\Z_t$ determine the conjugacy class of $(A,e)$ in
$W(t,m)$. Conversely, a collection of~$t$ arbitrary partitions $\{\lambda_z \mid z\in\Z_t\}$ determines a conjugacy class of $W(t,m)$ if and only if the total sum of the sizes of the
partitions~$\lambda_z$ is equal to~$m$.

Let~$p(d)$ denote the number of partitions of~$d$, and let~$P(u)$ be the power series
\[
P(u) =  \sum_{d=0}^{\infty} p(d)u^d.
\]
Consider the formal series
\[
Q(u) = \prod_{t=1}^{\infty}P(u^t)^t.
\]
From the discussion above, it is easily seen that each monomial term of degree~$tm$ in the expansion of $P(u^t)^t$ corresponds to a conjugacy class of $W(t,m)$, and that we therefore have
\[
P(u^t)^t = \sum_{m=0}^{\infty} k(W(t,m))u^{tm}.
\]
Now any single term in the
expansion of~$Q(u)$ corresponds to a choice, firstly of parameters~$m_t$ such that $\sum_t tm_t$ is finite, and secondly of a conjugacy class of $W(t,m_t)$ for each~$t$.
It follows from (\ref{product}) that each term of degree~$n$ in this expansion corresponds to a conjugacy class of $\Cent_{S_n}(g)$, where~$g$ is an element of~$S_n$ with~$m_t$~cycles of length~$t$. Now by (\ref{formula}) we have the formal identity
\[
Q(u) = \sum_{n=0}^{\infty} \frac{T(n)}{n!}u^n.
\]
Thus $Q(u)$ is equal to the right-hand side of (\ref{identity}), and it remains only to show that~$Q(u)$ is also equal to the left-hand side.

We use the Eulerian expansion of~$P(u)$,
\[
P(u) = \prod_{s=1}^{\infty} (1-u^s)^{-1}.
\]
From this it follows that
\[
Q(u) = \prod_{t=1}^{\infty}\prod_{s=1}^{\infty}(1-u^{st})^{-t} = \prod_{j=1}^{\infty}\prod_{t | j}(1-u^j)^{-t} = \prod_{j=1}^{\infty}(1-u^j)^{-\sigma(j)},
\]
as required.

Finally, I am indebted to Mark Wildon for the observation that both sides of (\ref{identity}) are convergent in the open unit disc
$|u|<1$, and that they therefore represent a complex function which is analytic in this disc. This can be seen by expressing the formal logarithm of the left-hand side of (\ref{identity})
as
\[
\sum_{j=1}^\infty\sum_{k=1}^\infty \frac{\sigma(j)}{k} u^{jk} = \sum_{d=1}^{\infty} \left(\sum_{a|d} \frac{a\sigma(a)}{d}\right) u^d,
\]
which has radius of convergence $1$, since clearly
\[
\sum_{a|d}a\sigma(a) < d^4.
\]
Thus the left-hand side of (\ref{identity}) represents an analytic function on the disc $|u|<1$,
and it follows that the right-hand side is the Taylor series of that function. An immediate consequence of this observation is that the growth of $T(n)/n!$ is subexponential;
I do not know of an easy combinatorial proof of this fact.

\affiliationone{Heilbronn Institute for Mathematical Research\\ School of Mathematics\\
University of Bristol\\ University Walk\\ Bristol, BS8 1TW\\ United Kingdom}

\end{document}